\documentclass[final]{siamltex}
\usepackage{amsmath}                                                            
\usepackage{amsfonts}
\usepackage{nccmath}     
\usepackage{amssymb}  
\usepackage[pdftex]{graphicx}
\usepackage{subfigure}
\usepackage{algorithm,algorithmicx,algpseudocode}
\usepackage{color}

\newcommand{\esssup}{\mathop{\mathrm{ess\,sup}}}

\newenvironment{remark}[1][remark]{Remark}

\begin{document}

\title{An a posteriori error estimator for discontinuous Galerkin methods \\ for non-stationary convection-diffusion problems}
\author{%
{\sc
Andrea Cangiani,
Emmanuil H.~Georgoulis and Stephen Metcalfe} \\[2pt]
Department of Mathematics, University of Leicester, University Road, Leicester LE1 7RH, United Kingdom}

\maketitle

\begin{abstract}
{This work is concerned with the derivation of a robust a posteriori error estimator for a discontinuous Galerkin method discretisation of linear non-stationary convection-diffusion initial/boundary value problems and with the implementation of a corresponding adaptive algorithm. More specifically, we derive a posteriori bounds for the error in the $L^2(H^1)$-type norm for an interior penalty discontinuous Galerkin (dG) discretisation in space and a backward Euler discretisation in time. An important feature of the estimator is robustness with respect to the P\'{e}clet number of the problem which is verified in practice by a series of numerical experiments. Finally, an adaptive algorithm is proposed utilising the error estimator. Optimal rate of convergence of the adaptive algorithm is observed in a number of test problems. 
}
{discontinuous Galerkin methods;  a posteriori error estimators; time-dependent convection-diffusion equation.}
\end{abstract}

\begin{section}{Introduction}
The interaction between convection and diffusion modelled by initial/boundary value problems involving partial differential equations (PDEs) poses a number of challenges in the context of their numerical approximation. Indeed, stationary convection-dominated convection-diffusion type problems admit analytical solutions of a multiscale nature that can contain steep gradients, usually termed boundary or interior layers, depending upon their location in the computational domain. The accurate and efficient numerical resolution of such steep layers is a challenge as their exact location cannot be, in general, known a priori. In the special cases where the location of boundary or interior layers is known, structured grids have been successfully employed (\cite{RST08}). Ultimately, in non-stationary convection-diffusion equations, the nature of the solution (including layers) may vary throughout the domain as time progresses. This renders the use of adaptive algorithms an attractive proposition for the accurate and efficient numerical approximation of this class of problems. As adaptive algorithms are usually based on suitable a posteriori error estimators, the formulation of estimators that can robustly estimate both the temporal and spatial nature of the error are of particular interest.

A posteriori error estimation for stationary linear equations is now relatively well understood; for pure diffusion problems there is a huge array of  estimators available for a wide variety of different types of finite element discretisations (\cite{V96,AO00}) and for dG methods in particular (\cite{KP03,BHL03,HSW07}). For stationary convection-diffusion equations, the quest for robust a posteriori error estimators (in the sense that they are independent of the P\'{e}clet number of the problem) has seen recent advancements in various contexts (\cite{V98,V05II,K03,S08,SZ09,SZ11}).

A posteriori error estimators for non-stationary linear convection-diffusion equations are also available for various discretisations (\cite{HS01,BC04,V05,EP05, ABR05,APS05,SW06,GHH07,LPP09}).  A posteriori error estimators for (spatial) dG methods for non-stationary pure diffusion problems can be found in \cite{GLV11,EV10,S12}.

This work is concerned with the derivation and implementation of an a posteriori estimator  in the $L^2(H^1)$-norm for the backward Euler discretisation in time (zero-th order dG method) and the interior penalty dG discretisation in space of the non-stationary linear convection-diffusion equation with variable coefficients. The derivation of the a posteriori bound utilises the elliptic reconstruction technique (\cite{MN03,GLV11}) which allows the use of the robust elliptic error estimator from \cite{SZ09}. The temporal residual is also treated so as to ensure that robustness with respect to the P\'{e}clet number of the problem is maintained. The resulting a posteriori error estimator is shown numerically to be reliable in both time and space.  The a posteriori estimator derived below can be viewed as the analogue of the one presented in \cite{V05} when the spatial discretisation is a dG method.

The a posteriori estimator is used to drive an adaptive algorithm which is numerically assessed in two ways. First, the savings with respect to the degrees of freedom by using space-time adaptivity are shown in a series of test problems. Second, the rates of convergence of the space-time adaptive algorithm are computed. The adaptive algorithm appears to convergence optimally in both space and time when the spatial polynomial degree, $p$, satisfies $p \geq 2$. Somewhat surprisingly, for $p=1$, the a posteriori estimator appears to converge at a suboptimal rate when the mesh changes between time steps. The nature of this degeneracy for $p=1$ remains unclear to the authors: it is in fact possible that the method itself is suboptimal in time when mesh change is allowed. It is also unclear if this is a special phenomenon of the discretisation presented in this work or a more general phenomenon extending to other discretisations of  parabolic problems. This phenomenon will be investigated further elsewhere.

The remainder of this work is structured as follows. The function space setting and the model problem are given in Section \ref{prelim}. Section \ref{dg_section} describes the discretisation of the problem. In Section \ref{ell_bounds_sec} we state an a posteriori error bound for the stationary elliptic problem. Sections \ref{sd_apost_sec} and \ref{fd_apost_sec} contain the derivation of the a posteriori bounds for the non-stationary convection-diffusion model problem in the semi-discrete and the fully-discrete settings, respectively. Section \ref{adaptive_sec} describes a space-time adaptive algorithm driven by the a posteriori estimators derived in the previous section. A series of numerical experiments are presented in Section \ref{numerics_sec} and the final conclusions are drawn in Section \ref{conclusions_sec}.
\end{section}

\begin{section}{Model problem}\label{prelim}
Let the computational domain $\Omega \subset \mathbb{R}^2$ be a bounded Lipschitz polygon with boundary $\partial\Omega$. 

We denote the standard $L^2$-inner product on $\Omega$ by $(\cdot,\cdot)$ and the standard $L^2$-norm on $\Omega$ by $||\cdot||$.  For $1\le p\le +\infty$, we define the spaces $L^p(0,T;X)$ (where $X$ is a real Banach space with norm $\|\cdot\|_X$) that consist of all measurable functions $v: [0,T]\to X$ for which:
\begin{equation}
\begin{aligned}
\|v\|_{L^p(0,T;X)}&:=\Big(\int_0^T \|v(t)\|_{X}^p d t\Big)^{1/p}<+\infty,\quad \text{for}\quad 1\le p< +\infty, \\
\|v\|_{L^{\infty}(0,T;X)}&:=\esssup_{0\le t\le T}\|v(t)\|_{X}<+\infty,\quad \text{for}\quad  p = +\infty.
\end{aligned}
\end{equation}
We also define $H^1(0,T,X):=\{u\in L^2(0,T;X): u_t\in  L^2(0,T;X)\}$. 
Finally, we denote by $C(0,T;X)$ and $C^{0,1}(0,T;X)$, respectively, the spaces of continuous and Lipschitz-continuous functions $v:[0,T] \rightarrow X$ such that:
\begin{equation}
\begin{gathered}
\displaystyle ||v||_{C(0,T;X)}=\max_{0 \leq t \leq T}{||v(t)||_X} < \infty \mbox{,} \\ ||v||_{C^{0,1}(0,T;X)}=\max \Big\{  ||v||_{C(0,T;X)}, ||\frac{\partial{v}}{\partial{t}}||_{C(0,T,X)} \Big\} < \infty \mbox{.}
\end{gathered}
\end{equation}
We consider the model problem of finding $u:\Omega\to\mathbb{R}$ such that:
\begin{equation}\label{model_strong}
\begin{aligned}
\frac{\partial{u}}{\partial{t}} - \varepsilon\Delta{u}+ {\bf a} \cdot \nabla{u}+bu &= f \mbox{ } &\text{ in }  \Omega\times(0,T]  \mbox{,} \\ u &=0 \mbox{ } &\text{ on }  \partial\Omega\times (0,T] \mbox{,} \\ u(\cdot,0) &=u_0 \mbox{ } &\text{ in } \bar{\Omega}\mbox{.}
\end{aligned}
\end{equation}
We make the following assumptions: $u_0 \in L^2(\Omega)$, $f \in C(0,T;L^2(\Omega))$, $\varepsilon\in(0,1]$, ${\bf a} \in C(0,T;W^{1, \infty}(\Omega))^2$, and $b \in C(0,T;L^{\infty}(\Omega))$.

It is assumed that the length of $\bf a\equiv\bf a(x,t)$ and the area of $\Omega$ are of order one, possibly up to rescaling, so that $\epsilon^{-1}$ can be taken to be the P\'{e}clet number of the problem. For simplicity, we assume that there are constants $\beta \geq 0$ and $c_* \geq 0$ such that:
\begin{equation}
b - \frac{1}{2}\nabla \cdot {\bf a} \geq \beta \quad \text{a.e. in }  \Omega \times [0,T],\qquad ||-\nabla\cdot {\bf a} + b||_{C(0,T;L^{\infty}(\Omega))} \leq c_*\beta\mbox{.}
\end{equation}
The weak form of \eqref{model_strong} then reads: find $u\in L^2(0,T;H^1_0(\Omega))\cap H^1(0,T;L^2(\Omega))$ such that for each $t \in (0,T]$ we have
\begin{equation}\label{model_weak}
\int_{\Omega} \! \frac{\partial{u}}{\partial{t}}v \, dx+\int_{\Omega} \! (\varepsilon \nabla{u} \cdot \nabla{v}+{\bf a} \cdot \nabla{u}v + buv) \, dx = \int_{\Omega} \! fv \, dx \quad \forall v \in H^1_0(\Omega)\mbox{.}
\end{equation}
Under the regularity assumptions above, we have that $u \in C(0,T;L^2(\Omega))$.

\end{section}

\begin{section}{Discontinuous Galerkin method}\label{dg_section}

Let the mesh $\zeta=\{K\}$ be a shape-regular subdivision of $\Omega$, with $K$ denoting a generic element. We assume that the subdivision $\zeta$ is constructed via affine mappings $F_{K}:\hat{K}\to K$ with non-singular Jacobian where $\hat{K}$ is the reference triangle or the reference square. The mesh is allowed to contain a uniformly fixed number of regular hanging nodes per edge. We define the finite element space
\begin{equation}\label{eq:FEspace}
V_h\equiv V_h(\zeta) := \{v \in L^2(\Omega): \mbox{ }v|_K\circ F_K \in \mathcal{P}^p(\hat{K}), K \in \zeta \},
\end{equation}
 where $\mathcal{P}^p(K)$ is the space of polynomials of total degree $p$ if $\hat{K}$ is the reference triangle, or the space of polynomials of degree $p$ in each variable if $\hat{K}$ is the reference square. 

Denote by $\mathcal{E}(\zeta)$ the set of all edges in the triangulation $\zeta$ and $\mathcal{E}^{int}(\zeta)\subset\mathcal{E}(\zeta)$ the subset of all interior edges. We also denote the diameter of an element $K$ by $h_K$ and the length of an edge $E$ by $h_E$. The outward unit normal on the boundary $\partial{K}$ of an element $K$ is denoted by ${\bf n}_K$. Given an edge $E\in \mathcal{E}^{int}(\zeta)$ shared by two elements $K$ and $K'$, a vector field ${\bf v}\in [H^{1/2}(\Omega)]^2$ and a scalar field $v\in H^{1/2}(\Omega)$, we define jumps and averages of ${\bf v}$ and $v$ across $E$ by:
\begin{equation}
\begin{aligned}
 \{{\bf v}\} & = \frac{1}{2}({\bf v}|_{\bar{K}}+{\bf v}|_{\bar{K}'}), \qquad &
   [{\bf v}]  =& {\bf v}|_{\bar{K}} \cdot {\bf n}_K+ {\bf v}|_{\bar{K}'} \cdot {\bf n}_{K'},  \\
  \{v\} & = \frac{1}{2}( v|_{\bar{K}}+ v|_{\bar{K}'}), \qquad &
   [v]   = &  v|_{\bar{K}}  {\bf n}_K+  v|_{\bar{K}'}  {\bf n}_{K'}.
 \end{aligned}
\end{equation}
If $E\subset \partial\Omega$, we set $\{{\bf v}\}={\bf v}$, $[{\bf v}]={\bf v} \cdot {\bf n}$, $\{v\}= v$ and $[ v]= v {\bf n}$, with ${\bf n}$ denoting the outward unit normal to the boundary $\partial\Omega$. 

We define the inflow and outflow parts of the boundary $\partial\Omega$ at the time $t$, respectively, by:
\begin{equation}
 \partial\Omega^t_{in} = \{x \in \partial\Omega \mbox{ }|\mbox{ } {\bf a}(x,t) \cdot {\bf n}(x) < 0 \}\mbox{,} \qquad \partial\Omega^t_{out} = \{x \in \partial\Omega \mbox{ } | \mbox{ } {\bf a}(x,t) \cdot {\bf n}(x) \geq 0 \}\mbox{.}
\end{equation}
The inflow and outflow parts of an element $K$ at time $t$ are similarly defined as:
\begin{equation}
\partial K^t_{in} = \{x \in \partial K \mbox{ } | \mbox { } {\bf a}(x,t) \cdot {\bf n}_K(x) < 0 \}\mbox{,} \qquad \partial K^t_{out} = \{x \in \partial K \mbox{ } | \mbox{ } {\bf a}(x,t) \cdot {\bf n}_K(x) \geq 0 \}\mbox{.}
\end{equation}

The semi-discrete discontinuous Galerkin approximation to \eqref{model_weak} then reads as follows. For $t=0$, set $ u_h(0) \in V_h$ to be a projection of $u_0$ onto $V_h$. For $t \in (0,T]$, find $u_h \in C^{0,1}(0,T;V_h)$ such that
\begin{equation}\label{dg_sd}
(\frac{\partial{u_h}}{\partial{t}},v_h)+B(t;u_h,v_h)+K_h(u_h,v_h)=(f,v_h) \qquad  \forall v_h \in V_h \mbox{,}
\end{equation} 
with
\begin{equation}
\begin{aligned}
B(t;w,v) \equiv & B(\zeta;t;w,v)\\
:= & \sum_{K \in \zeta}{\int_{K} \! 
(\varepsilon \nabla w  -{\bf{a}}w) \cdot \nabla v+ (b- \nabla \cdot {\bf a})wv\, dx}
+\sum_{E \in \mathcal{E}(\zeta)}\frac{\varepsilon \gamma}{h_E}\int_{E} \! [w] \cdot [v] \, ds\\ 
&+ \sum_{K \in \zeta}\bigg(\int_{\partial{K}^t_{out} \cap \partial\Omega^t_{out}} \! {\bf{a}} \cdot {\bf n}_K wv \, ds+\int_{\partial{K}^t_{out} \setminus \partial\Omega} \! {\bf{a}}\cdot {\bf n}_K w(v|_{\bar{K}}-v|_{\bar{K}'}) \, ds\bigg), \\ 
K_h(w,v)\equiv & K_h(\zeta;w,v):= -\sum_{E \in \mathcal{E}(\zeta)}\int_{E} \! \{\varepsilon \nabla w \} \cdot [v] + \{\varepsilon \nabla v \} \cdot [w] \, ds,
\end{aligned}
\end{equation}
where $\bar{K}\cap\bar{K}'=E\subset \partial{K}^t_{out} \setminus \partial\Omega$.  In standard fashion, the penalty parameter $\gamma>0$ is chosen large enough so that the operator $B+K_h$ is coercive. Moreover, for simplicity of the presentation, we assume $\gamma > 1$ so that the constants in the subsequent discussion are independent of it. 

We note that the bilinear form $K_h$ is not well-defined for arguments in $H^1_0(\Omega)$, but the bilinear form $B$ is and for $u,v \in H^1_0(\Omega)$ and $t \in (0,T]$, we have
\begin{equation}
B(t;u,v)=\int_{\Omega} \! (\varepsilon \nabla{u} \cdot \nabla{v}+{\bf a}\cdot \nabla{u}v + buv) \, dx.
\end{equation}
In light of this, the weak problem~\eqref{model_weak} can be rewritten for each $t \in (0,T]$ as
\begin{equation}
\label{weakform}
(\frac{\partial{u}}{\partial{t}},v)+B(t;u,v)=(f,v) \qquad \forall v \in H^1_0(\Omega).
\end{equation}

We shall also consider a full discretisation of problem \eqref{model_weak} by using a backward Euler method to approximate the time derivative. 

To this end, consider a subdivision of $[0,T]$ into time intervals of lengths $\tau_1, \tau_2, ... , \tau_n$ such that $\sum_{j=1}^n{\tau_j}=T$ for some $n \geq 1$ and set $t^0 = 0$ and $t^k := \sum_{j=1}^{k}\tau_{k}$. Denote an initial triangulation by $\zeta^0$. We associate to each time step $k>0$ a triangulation $\zeta^k$  which is assumed to have been obtained from $\zeta^{k-1}$ by locally refining and coarsening $\zeta^{k-1}$. This restriction upon mesh change is made to avoid degradation of the finite element solution, {\em cf.}~\cite{D82}. To each mesh $\zeta^{k}$, we assign the finite element space $V_h^k = V_h(\zeta^k)$ given by~\eqref{eq:FEspace}. 

We set $B^k(t;\cdot,\cdot):=B(\zeta^k;t;\cdot,\cdot)$ and $K_h^k(\cdot,\cdot):=K_h(\zeta^k;\cdot,\cdot)$. Moreover, when the last two arguments of $B$ are in $H^1_0(\Omega)$, we simply denote the bilinear form by $B$. We also denote $f(.,t^k)=f^k$, ${\bf a}(.,t^k) = {\bf a}^k$, and $b(.,t^k)=b^k$ for brevity. 

The fully-discrete dG method then reads as follows. Set $u_h^0$ to be a projection of $u_0$ onto $V_h^0$. For $k=0$,...,$n-1$, find $u_h^{k+1} \in V_h^{k+1}$ such that
\begin{equation}\label{dg_fd}
\begin{aligned}
(\frac{u_h^{k+1}-u_h^k}{\tau_{k+1}},v_h^{k+1})+B^{k+1}(t^{k+1};u_h^{k+1},v_h^{k+1})+K^{k+1}_h(u_h^{k+1},v_h^{k+1})&=(f^{k+1},v_h^{k+1}) \\&\forall v_h^{k+1} \in V_h^{k+1}.
\end{aligned}
\end{equation}
We shall take $u_h^0$ to be the orthogonal $L^2$-projection of $u_0$ onto $V_h^0$, although other projections onto $V_h^0$ can also be used.\end{section}

\begin{section}{Error bounds for the dG method for the stationary problem}\label{ell_bounds_sec}
To analyse the error, we introduce the following quantities:
\begin{equation}
\begin{aligned}
|||u|||_{\zeta} := & \bigg(\sum_{K \in \zeta}(\varepsilon||\nabla{u}||^2_{L^2(K)}+\beta||u||^2_{L^2(K)})+\sum_{E \in \mathcal{E}(\zeta)}{\frac{\varepsilon \gamma}{h_E}||[u]||^2_{L^2(E)}}\bigg)^{1/2}, 
\\ 
|u|_{A,\zeta} := & \bigg(\bigg(\sup_{v \in H^1_0(\Omega) \setminus \{0\}}{\frac{\int_{\Omega} \! {{\bf a}u \cdot \nabla{v}} \, dx}{|||v|||}}\bigg)^2 + \sum_{E \in \mathcal{E}(\zeta)}{\Big(\beta h_E + \frac{h_E}{\varepsilon}\Big)||[u]||^2_{L^2(E)}}\bigg)^{1/2}.
\end{aligned}
\end{equation}
We note that $|||\cdot|||_{\zeta}$ and $|\cdot|_{A,\zeta}$ define norms on $H^1_0(\Omega)+V_h$. When no confusion is likely to occur we shall suppress the dependence on the mesh $\zeta$ and write $|||\cdot|||$ and $|\cdot|_{A}$. 

Throughout this work, the symbols $\lesssim$ and $\gtrsim$ are used to denote inequalities true up to a positive constant independent of $\varepsilon$, $\beta$, $u$, and $u_h$. 

Let $t\in (0,T]$ be fixed. It is easy to see that the bilinear form $B(t;\cdot,\cdot)$ is coercive on $H^1_0(\Omega)$, viz.,
\begin{equation}\label{coercivity}
B(t;v,v) \geq |||v|||^2,
\end{equation}
for all  $v \in H^1_0(\Omega)$, and is continuous in the following sense:
\begin{equation}\label{continuity}
B(t;w,v) \lesssim (|||w|||+|w|_{A})|||v|||,
\end{equation}
for all $w \in H^1_0(\Omega) + V_h$ and  $v \in H^1_0(\Omega)$. Moreover, the discrete bilinear form is coercive for $v_h$ in $V_h$ with respect to the $|||\cdot|||$ norm, viz.,
\begin{equation}
\label{FEcoercive}
B(t;v_h,v_h) + K_h(v_h,v_h) \gtrsim |||v_h|||^2
\end{equation}
Next, we introduce the following notation which will be used to define the a posteriori estimators:
\begin{equation}
{\alpha}_K := \min (h_K{\varepsilon}^{-\frac{1}{2}},{\beta}^{-\frac{1}{2}}),\quad {\alpha}_E := \min (h_E{\varepsilon}^{-\frac{1}{2}},{\beta}^{-\frac{1}{2}}),\quad {\alpha}_T = \min({\varepsilon}^{-\frac{1}{2}},{\beta}^{-\frac{1}{2}}).
\end{equation}
An a posteriori estimator for the stationary problem inspired by \cite{SZ09} will be utilised in our analysis. More specifically, we have the following result whose proof is completely analogous to the one of Theorem 3.2 in \cite{SZ09} and is therefore omitted for brevity.

\begin{theorem}\label{elliptic_apost}
For a given $t \in (0,T]$, let $u^s$ be such that
\[
B(t;u^s,v)=(f,v)\qquad \forall v\in H^1_0(\Omega),
\]
and consider $u^s_h\in V_h$ such that
\[
B(t;u^s_h,v_h)+K_h(u^s_h,v_h)=(f,v_h) \qquad\forall v_h\in V_h.
\]
Then the following a posteriori bound holds:
\begin{equation}\label{apost_stationary}
\begin{aligned}
\notag
(||| u^s - u^s_h |||&+|u^s-u^s_h|_A)^2 \lesssim \sum_{K\in \zeta} {\alpha}^2_K ||f + \epsilon \Delta u^s_h-{\bf a} \cdot \nabla u^s_h - b u^s_h||_{L^2(K)}^2 \\
&+\sum_{E\in \mathcal{E}^{int}(\zeta)} \epsilon^{\frac{3}{2}}{\alpha}_E|| [\nabla u^s_h] ||_{L^2(E)}^2 
+ \sum_{E\in \mathcal{E}(\zeta)} \Big(\frac{\gamma \epsilon}{h_E}+\beta h_E+\frac{h_E}{\epsilon}\Big)|| [ u^s_h] ||_{L^2(E)}^2.
\end{aligned}
\end{equation}
\end{theorem}
  \end{section}
 
 \begin{section}{An a posteriori bound for the semi-discrete method}\label{sd_apost_sec}
 To highlight the main ideas, we begin by deriving an a posteriori bound for the semi-discrete problem.
 \begin{definition}
\label{reconstrdef1}
For each $t\in(0,T]$, we define \emph{the elliptic reconstruction $w \in H^1_0(\Omega)$} to be the (unique) solution of the problem
\begin{equation}\label{reconstruction_sd}
\notag
B(t;w,v)=(f-\frac{\partial{u_h}}{\partial{t}},v) \qquad \forall v \in H^1_0(\Omega).
\end{equation}
\end{definition}

\begin{remark}\label{reconstruction_remark}
The dG discretisation of Definition \ref{reconstrdef1} is to find a function $w_h \in C^{0,1}(0,T,V_h)$ such that for each $t \in (0,T]$ we have
\[
 B(t;w_h,v_h) + K_h(w_h,v_h)=(f-\frac{\partial{u_h}}{\partial{t}},v_h) \qquad \forall v_h \in V_h.
 \]
This, in conjunction with \eqref{dg_sd} and \eqref{FEcoercive}, implies that $w_h=u_h$. Thus, $|||w-u_h|||+|w-u_h|_{A}$ can be estimated using Theorem \ref{elliptic_apost}. 
\end{remark}

We decompose the error as follows:
 \begin{equation}
  e=u-u_h=\rho+\theta,\ \text{ with }\  \rho := u-w\  \text{ and }\ \theta := w-u_h.
 \end{equation}
 The dG solution, $u_h$, admits a decomposition into a conforming part $u_{h,c} \in H^1_0(\Omega) \cap V_h$ and a non-conforming part $u_{h,d} \in V_h$ with $u_h=u_{h,c}+u_{h,d}$, such that:
\begin{equation}\label{nonconforming_bound}
\begin{aligned}
|||u_{h,d}|||^2 +|u_{h,d}|_{A}^2 & \lesssim \sum_{E \in \mathcal{E}(\zeta)}{\Big(\frac{\varepsilon \gamma}{h_E} + \beta{h_E}+\frac{h_E}{\varepsilon}\Big)||[u_h]||^2_{L^2(E)}},
\\
||\frac{\partial{u_{h,d}}}{\partial{t}}||^2 & \lesssim \sum_{E \in \mathcal{E}(\zeta)}{h_E||[\frac{\partial{u_h}}{\partial{t}}]||^2_{L^2(E)}}.
\end{aligned}
\end{equation} 
We refer to \cite{SZ09} for proof of these estimates which are based on respective constructions by  \cite{KP03} and \cite{V05II}. We further define $e_c:=u-u_{h,c}$ and $\theta_c := w - u_{h,c}$.

\begin{lemma}\label{error_eq_simple}
For each $t\in(0,T]$ and for all $v \in H^1_0(\Omega)$ we have
\[
(\frac{\partial{e}}{\partial{t}},v)+B(t;\rho,v)=0.
\] 
\end{lemma}
\begin{proof}
This follows directly from Definition \ref{reconstrdef1} and  \eqref{weakform}.
\end{proof}

We define the \emph{initial condition estimator}, $\tilde{\eta}_I$, by
\begin{equation}
\tilde{ \eta}^2_I = ||u_0-u_h(0)||^2 + \sum_{E \in \mathcal{E}({\zeta})}{h_E||[u_h(0)]||^2_{L^2(E)}},
\end{equation}
while the \emph{spatial estimator}, $\tilde{\eta}_S$, is given by
\begin{equation}
\tilde{\eta}_S^2= \int_0^T \!\tilde{\eta}_{S_1}^2\, dt+\min\big\{\Big(\int_0^T \! \tilde{\eta}_{S_2} \, dt \Big)^2,{\alpha}^2_T\int_0^T \! \tilde{\eta}_{S_2}^2 \, dt\big\},
\end{equation}
where
\begin{equation}
\begin{aligned}
\tilde{\eta}_{S_1}^2 = & \sum_{K \in \zeta}{\alpha}^2_K||f-\frac{\partial{u_h}}{\partial{t}}+\varepsilon\Delta{u_h}-{\bf a}\cdot \nabla{u_h}-bu_h||^2_{L^2(K)}
+\sum_{E \in \mathcal{E}^{int}(\zeta)}{{\varepsilon}^{\frac{3}{2}}{\alpha}_E||[\nabla{u_h}]||^2_{L^2(E)}}\\
&+\sum_{E \in \mathcal{E}(\zeta)}{\Big(\frac{\varepsilon \gamma}{h_E} + \beta{h_E}+\frac{h_E}{\varepsilon}\Big)||[u_h]||^2_{L^2(E)}},
\end{aligned}
\end{equation}
is the residual term and 
\begin{equation}
 \tilde{\eta}_{S_2}^2=  \sum_{E \in \mathcal{E}(\zeta)}{h_E||[\frac{\partial{u_h}}{\partial{t}}]||^2_{L^2(E)}},
\end{equation}
is the rate of change of non-conformity. 
\begin{theorem}\label{apost_sd_thm}
 The error $e$ of the semi-discrete dG method~\eqref{dg_sd} satisfies the bound
\begin{equation}
\notag
\int_0^T \! |||e|||^2 \, dt \lesssim \tilde{\eta}^2_I+\tilde{\eta}^2_S.
\end{equation}
\end{theorem}

\begin{proof}  Choosing $v=e_c$ in Lemma \ref{error_eq_simple} gives
\begin{equation}
(\frac{\partial{e_c}}{\partial{t}},e_c)+B(t;e_c,e_c)=(\frac{\partial{u_{h,d}}}{\partial{t}},e_c)+B(t;\theta_c,e_c).
\end{equation}
Using \eqref{coercivity},~\eqref{continuity}, and Young's inequality we arrive to
\begin{equation}\label{first_sd_bound}
 \frac{d}{dt}(||e_c||^2)+|||e_c|||^2 \lesssim (|||\theta_c|||+|\theta_c|_{A})^2+||\frac{\partial{u_{h,d}}}{\partial{t}}||\, ||e_c||.
\end{equation}
Let $ E_c:=||e_c(T_0)||$ with $T_0\in[0,T]$ such that $||e_c(T_0)||= \max_{0 \leq t \leq T}{||e_c(t)||}$ then integrating on $[0,T]$ gives
\begin{equation}\label{l2h1_bound_sd}
 \int_0^T \! {|||e_c|||^2 \, dt}\lesssim ||e_c(0)||^2+\int_0^T \! {(|||\theta_c|||+|\theta_c|_{A})^2 \, dt}+E_c \int_0^T \! {||\frac{\partial{u_{h,d}}}{\partial{t}}|| \, dt}.
\end{equation}
Integrating \eqref{first_sd_bound} on $[0,T_0]$ and using Young's inequality 
once again gives
\begin{equation}\label{Ec_bound_sd}
 (E_c)^2 \lesssim ||e_c(0)||^2+\int_0^{T} \! {(|||\theta_c|||+|\theta_c|_{A})^2 \, dt}+ \Big(\int_0^{T} \! {||\frac{\partial{u_{h,d}}}{\partial{t}}|| \, dt}\Big)^2.
\end{equation}
Using \eqref{Ec_bound_sd} to bound $E_c$ on the right-hand side of \eqref{l2h1_bound_sd} yields
\begin{equation}\label{error_eq_first}
 \int_0^T \! {|||e_c|||^2 \, dt}\lesssim ||e_c(0)||^2+\int_0^T \! {(|||\theta_c|||+|\theta_c|_{A})^2 \, dt}+ \Big(\int_0^T \! {||\frac{\partial{u_{h,d}}}{\partial{t}}|| \, dt \Big)^2}.
\end{equation}
Going back to \eqref{first_sd_bound}, using the Poincar\'{e}-Friedrichs inequality on $||e_c||$ and working as above gives
\begin{equation}\label{error_eq_second}
 \int_0^T \! {|||e_c|||^2 \, dt}\lesssim ||e_c(0)||^2+\int_0^T \! {(|||\theta_c|||+|\theta_c|_{A})^2 \, dt}+ {\alpha}_T^2\int_0^T \! {||\frac{\partial{u_{h,d}}}{\partial{t}}||^2 \, dt}.
\end{equation}
Using the triangle inequality, we have
\begin{equation}\label{triangle_last}
|||e||| \le |||e_c|||+|||u_{h,d}|||\quad \text{ and }\quad |||\theta_c|||+|\theta_c|_{A} \le |||\theta|||+|\theta|_{A}+|||u_{h,d}|||+|u_{h,d}|_{A}.
\end{equation}
Combining \eqref{error_eq_first}, \eqref{error_eq_second}, and \eqref{triangle_last} yields
\begin{equation}
\begin{aligned}
 \int_0^T \! {|||e|||^2 \, dt}  \lesssim &\ ||e_c(0)||^2+\int_0^T \! {(|||\theta|||+|\theta|_{A})^2 \, dt}+\int_0^T \! {(|||u_{h,d}|||^2+|u_{h,d}|^2_{A}) \, dt} \\  & + \min\Big\{\Big(\int_0^T \! {||\frac{\partial{u_{h,d}}}{\partial{t}}|| \, dt\Big)^2},{\alpha}_T^2\int_0^T \! {||\frac{\partial{u_{h,d}}}{\partial{t}}||^2 \, dt}\Big\}.
\end{aligned}
\end{equation}
The proof then follows directly from Theorem \ref{elliptic_apost} and \eqref{nonconforming_bound}.

\end{proof}

\end{section}

\begin{section}{An a posteriori bound for the fully-discrete method}\label{fd_apost_sec}
We now continue by applying to the fully-discrete setting the general framework presented in the previous section.

\begin{definition}\label{ellipticreconstr2} We define the \emph{elliptic reconstruction} $w^{k+1} \in H^1_0(\Omega)$ to be the unique solution of the elliptic problem
\begin{equation}\label{ell_rec_fd}
\notag
B(t^{k+1};w^{k+1},v)=(f^{k+1}-\frac{u_h^{k+1}-u_h^k}{\tau_{k+1}},v) \qquad
\forall v \in H^1_0(\Omega).
\end{equation}
\end{definition}

At each time step $k$, we decompose the dG solution $u_h^k$ into a conforming part $u_{h,c}^k \in H^1_0(\Omega) \cap V_h^k$ and a non-conforming part $u_{h,d}^k \in V_h^k$ such that $u_h^k = u_{h,c}^k + u_{h,d}^k$. Given $t \in (t^k,t^{k+1}]$, we (re)define $u_h(t)$ to be the linear interpolant with respect to $t$ of the values $u_h^k$ and $u_h^{k+1}$, viz.,
\begin{equation}
u_h(t):=l_k(t)u_h^k+l_{k+1}(t)u_h^{k+1},
\end{equation}
where $\{l_k, l_{k+1}\}$ denotes the standard linear Lagrange interpolation basis defined on the interval $[t^k,t^{k+1}]$.
 We define $u_{h,c}(t)$ and $u_{h,d}(t)$ analogously. We can then decompose the error $e=u-u_h=e_c-u_{h,d}$ where $e_c=u-u_{h,c}$. It will also be useful to define $\theta^{k} = w^{k}-u_h^{k} = \theta_c^{k}-u_{h,d}^{k}$.

\begin{lemma}\label{error_eq_simple_fd}
Given $t \in (t^k,t^{k+1}]$ we have 
\begin{equation}\label{error_relation_fd}
\notag
(\frac{\partial{e}}{\partial{t}},v)+B(t;e_c,v)=B(t^{k+1};w^{k+1},v)-B(t,u_{h,c},v)+(f-f^{k+1},v) \qquad \forall v \in H^1_0(\Omega).
\end{equation} 
\end{lemma}
\begin{proof}
This follows from Definition \ref{ellipticreconstr2} and \eqref{weakform} and by rearranging the resulting equation. 
\end{proof}

\noindent Before proving the a posteriori bounds for the fully-discrete method, we introduce the error estimators. We begin by defining the \emph{initial condition estimator}, $\eta_I$, by
\begin{equation}
 \eta^2_I = ||u_0-u_h^0||^2 + \sum_{E \in \mathcal{E}({\zeta^0})}{h_E||[u_h^0]||^2_{L^2(E)}},
\end{equation}
while the \emph{spatial estimator}, $\eta_S$, is given by
\begin{equation}
\eta^2_S =  \sum_{j=0}^{n-1}\tau_{j+1} \eta_{S_1,j+1}^2+ \min\bigg\{\Big(\sum_{j=0}^{n-1} \int_{t^j}^{t^{j+1}} \! \eta_{S_2,j+1} \, dt \bigg)^2, {\alpha}_T^2 \sum_{j=0}^{n-1} \int_{t^j}^{t^{j+1}} \! \eta_{S_2,j+1}^2 \, dt \Big\},
\end{equation}
where
\begin{equation}
\begin{aligned}
 \eta_{S_1,j+1}^2 &=\hspace{-2mm} \sum_{K \in \zeta^{j}\cup\zeta^{j+1}}{{\alpha}_K^2||f^{j+1}-\frac{u_h^{j+1}-u_h^j}{\tau_{j+1}}+\varepsilon\Delta u_h^{j+1}}-{\bf a}^{j+1} \cdot \nabla u_h^{j+1}-b^{j+1}u_h^{j+1}||_{L^2(K)}^2
\\
&\hspace{-8mm}+\hspace{-2mm}\sum_{E \in \mathcal{E}^{int}(\zeta^{j}\cup\zeta^{j+1})}{\varepsilon}^{\frac{3}{2}}{\alpha}_E ||[\nabla u_h^{j+1}]||^2_{L^2(E)}
+
 \sum_{E \in \mathcal{E}(\zeta^{j}\cup\zeta^{j+1})}\Big(\frac{\gamma\varepsilon}{h_E}+\beta h_E +\frac{h_E }{\varepsilon}\Big)||[u_h^{j+1}]||^2_{L^2(E)}
\\
&\hspace{-8mm}+\hspace{-2mm}\sum_{E \in \mathcal{E}(\zeta^j \cup\zeta^{j+1})}\Big(\frac{\gamma\varepsilon}{h_E}+\beta h_E +\frac{h_E }{\varepsilon}\Big)||[u_h^{j+1}-u_h^j]||^2_{L^2(E)} ,
\end{aligned}
\end{equation}
and
\begin{equation}
\begin{aligned}
 \eta_{S_2,j+1}^2 = & \sum_{E \in \mathcal{E}(\zeta^j \cup \zeta^{j+1})}h_E||[\frac{u_h^{j+1}-u_h^j}{\tau_{j+1}}]||^2_{L^2(E)}\\
 +
 & \sum_{E \in \mathcal{E}^{int}(\zeta^j \cup \zeta^{j+1})} h_E^{-1} ||[l_{j}(t){\bf a}(u_h^{j+1}-u_h^j)+({\bf a}^{j+1}-{\bf a})u_h^{j+1}]||^2_{L^2(E)}.
\end{aligned}
\end{equation}
The final term is the \emph{time (or temporal) estimator}, $\eta_T$, given by 
\begin{equation}
\eta^2_T = \frac{1}{4}\sum_{j=0}^{n-1} \tau_{j+1} \eta_{T_1,j+1}^2  + \min\Big\{\bigg(\sum_{j=0}^{n-1} \int_{t^j}^{t^{j+1}} \! \eta_{T_2,j+1} \, dt \bigg)^2, {\alpha}_T^2 \sum_{j=0}^{n-1} \int_{t^j}^{t^{j+1}} \! \eta_{T_2,j+1}^2\, dt \Big\}.
\end{equation}
where
\begin{equation}
 \eta_{T_1,j+1}^2 =  \sum_{K \in \zeta^j \cup \zeta^{j+1}}\epsilon||\nabla(u_h^{j+1}-u_h^j)||^2_{L^2(K)},
\end{equation}
and
\begin{equation}
\begin{aligned}
 \eta_{T_2,j+1}^2 =   \sum_{K \in \zeta^j \cup \zeta^{j+1}}||&l_{j}(t)({\bf a}\cdot \nabla(u_h^{j+1}-u_h^j)+b(u_h^{j+1}-u_h^j)) +f-f^{j+1}\\
&+({\bf a}^{j+1}-{\bf a})\cdot \nabla{u_h^{j+1}}+(b^{j+1}-b)u_h^{j+1}||^2_{L^2(K)}.
\end{aligned}
\end{equation}

\begin{theorem}\label{apost_fd_thm}
 The error $e$ of the fully-discrete method satisfies the bound
\begin{equation}
\sum_{j=0}^{n-1} \int_{t^j}^{t^{j+1}}\!|||e|||^2 \, dt \lesssim \eta^2_I+\eta^2_S+\eta^2_T.
\end{equation}
\end{theorem}
\begin{proof} From Lemma \ref{error_eq_simple_fd} we have
\begin{equation}
(\frac{\partial{e_c}}{\partial{t}},v)+B(t;e_c,v)=(\frac{\partial{u_{h,d}}}{\partial{t}},v)+B(t^{k+1};w^{k+1},v)-B(t,u_{h,c},v)+(f-f^{k+1},v),
\end{equation}
which upon straightforward manipulation and setting $v=e_c$ gives
\begin{equation}\label{error_eq_fd}
\begin{aligned}
\frac{d}{dt}(||e_c||^2)+B(t;e_c,e_c)= &(\frac{\partial{u_{h,d}}}{\partial{t}},e_c)+B(t^{k+1};\theta_c^{k+1},e_c)+l_{k}(t)B(t;u_{h,c}^{k+1}-u_{h,c}^k,e_c)\\ &+B(t^{k+1};u_{h,c}^{k+1},e_c)-B(t;u^{k+1}_{h,c},e_c)+(f-f^{k+1},e_c).
\end{aligned}
\end{equation}
Denoting by $I^k_he_c$ the $L^2$-projection of $e_c$ onto $H^1_0(\Omega) \cap V_h(\zeta^k \cup \zeta^{k+1})$, using \eqref{coercivity}, \eqref{continuity}, the Cauchy-Schwarz inequality, Young's inequality and the triangle inequality yields
\begin{equation}
\begin{aligned}
 \frac{d}{dt}(||e_c||^2)+|||e_c|||^2  \lesssim &\ ||\frac{\partial{u_{h,d}}}{\partial{t}}||||e_c||+(|||\theta^{k+1}|||+|\theta^{k+1}|_{A})^2+|||u_{h,d}|||^2+|u_{h,d}|_{A}^2\\&+T_1 +  T_2 + T_3 + T_4 + T_5,
\end{aligned}
\end{equation}
where
\begin{equation}
\begin{aligned}
T_1 :=\sum_{K \in \zeta^k \cup \zeta^{k+1}} & \int_K \! (f - f^{k+1} + ({\bf a}^{k+1}-{\bf a})\cdot \nabla u_h^{k+1} +(b^{k+1}-b)u_h^{k+1} \\ &+ l_{k}(t)({\bf a}\cdot \nabla(u_h^{k+1}-u_h^k)+b(u_h^{k+1}-u_h^k)))e_c \, dx,
\end{aligned}
\end{equation}
\begin{equation}
 T_2 :=l_{k}(t)\sum_{K \in \zeta^k \cup \zeta^{k+1}} \int_K \! \epsilon \nabla(u_h^{k+1}-u_h^k) \cdot \nabla{e_c} \, dx,
\end{equation}
\begin{equation}
 T_3 := \frac{1}{2}\sum_{K \in \zeta^k \cup \zeta^{k+1}} \int_{\partial K \backslash \partial\Omega} \! [({\bf a}^{k+1}-{\bf a})u_h^{k+1}](e_c - I^k_he_c) \, ds,
\end{equation}
\begin{equation}
 T_4 :=\frac{1}{2}l_{k}(t)\sum_{K \in \zeta^k \cup \zeta^{k+1}} \int_{\partial K \backslash \partial\Omega} \! [{\bf a}(u_h^{k+1}-u_h^k)](e_c - I^k_he_c) \, ds,
\end{equation}
and
\begin{equation}
 T_5 := \frac{1}{2}\sum_{K \in \zeta^k \cup \zeta^{k+1}} \int_{\partial K \backslash \partial\Omega} \! [l_{k}(t){\bf a}(u_h^{k+1}-u_h^k)+({\bf a}^{k+1}-{\bf a})u_h^{k+1}]I^k_he_c \, ds.
\end{equation}
Using the discrete and integral versions of the Cauchy-Schwarz inequality, we immediately have
\begin{equation}
T_1\le \eta_{T_2,k+1} ||e_c||,\qquad T_2\le \eta_{T_1,k+1} |||e_c|||.
\end{equation}
Further, using the approximation properties of $I_h^k$ we have
\begin{equation}
T_3+T_4\lesssim \eta_{S_1,k+1} |||e_c|||.
\end{equation}
Next, using the standard inverse estimate $||I_h^k e_c||_{L^2(E)}^2\lesssim h_E^{-1} ||I_h^k e_c||_{L^2(K)}^2$, for $E\subset \partial K$, and the stability of the $L^2$-projection $|| I_h^k e_c|| \le ||e_c||$, we deduce
\begin{equation}
T_5 \lesssim \eta_{S_2,k+1} || e_c ||.
\end{equation}
The proof then follows from Theorem \ref{elliptic_apost} and \eqref{nonconforming_bound} and by employing a bounding strategy identical to that used in Theorem \ref{apost_sd_thm}.
 \end{proof}
 
\begin{remark}
An alternative a posteriori bound can be show, using directly \eqref{continuity} on \eqref{error_eq_fd}. The resulting bound will contain computable terms in the $|\cdot |_A$-norm, which can, in turn, be estimated by a more easily localisable term of the form $\epsilon^{-1/2}||\cdot||_{L^2(\Omega)}$. Numerical experiments (omitted here for brevity) show that this alternative estimator produces very similar results to the ones based on Theorem \ref{apost_fd_thm}.
\end{remark} 

\begin{remark}
Although efficiency bounds are not presented in this work, we refer to \cite{V05} for the proof of efficiency bounds for the a posteriori estimator presented therein for the SUPG spatial setting which should be extendable to the estimator presented in Theorem \ref{apost_fd_thm} given the similarity of some of the terms presented.
\end{remark}

\begin{remark}
The use of elliptic reconstruction is not essential for the proof of Theorem \ref{apost_sd_thm} and Theorem \ref{apost_fd_thm}. It is possible to derive the residual based a posteriori bounds directly albeit at the cost of a lengthier calculation. The advantage of using the elliptic reconstruction in the proof lies in the fact that it can be easily modified to include non-residual based spatial a posteriori estimators. This, in turn, may offer improvements in robustness with respect to the P\'eclet number (cf. \cite{S08}).
\end{remark}
 
\end{section}

\begin{section}{An adaptive algorithm}\label{adaptive_sec}
The a posteriori bounds presented above will be used to drive a space-time adaptive algorithm. A number of adaptive algorithms for parabolic problems have been proposed in the literature; see e.g.~\cite{CF04,SS05} and the references therein. 

Here, we propose a variant of the adaptive algorithms from \cite{CF04,SS05} which appears to perform well for our discretisation. The pseudocode is given in Algorithm 1.
The algorithm is based on using different parts of the a posteriori estimator from Theorem \eqref{apost_fd_thm} to drive space-time adaptivity. 

\begin{algorithm} \label{adaptive_algorithm}
  \begin{algorithmic}[1]
     \State {\bf Input:} $\epsilon, {\bf a}, b, f, u_0, T, \Omega, m, n, \zeta^0, \gamma, {\tt initol},{\tt ttol}, {\tt stola}, {\tt stolb}, {\tt ref\%}, {\tt coar\%}$.
     \State Calculate $u_h^0$.
     \State  Set $j = 0$, $\tau_1, ..., \tau_n = T/n$, $time = \tau_1$, $threshold = T/m$, $counter = 0$.
    \While {$\eta_I > initol$}
    \State Refine $\zeta^0$ by 10\% and coarsen $\zeta^0$ by 5\%.
     \State Calculate $u_h^0$.
    \EndWhile
    \While {  $time \leq  T$}
    \State Calculate $u_h^{j+1}$ from $u_h^j$.
    \While
     {$\hat{\eta}_{T,j+1}>$ {\tt ttol}}
     \State $n \leftarrow n+1$

 \State $time \leftarrow time - \tau_{j+1}$

\State $\tau_{j+3} = \tau_{j+2}, ..., \tau_{n} = \tau_{n-1}$

 \State $\tau_{j+2} = \tau_{j+1}/2$

 \State $\tau_{j+1} \leftarrow \tau_{j+1}/2$

 \State $\time \leftarrow time + \tau_{j+1}$

\State Calculate $u_h^{j+1}$ from $u_h^j$.
    \EndWhile
    \State $counter \leftarrow counter + \tau_{j+1}$
    \If {$\eta_{S_1,j+1} > {\tt stola}$ and $counter > threshold$}
    \State Create $\zeta^{j+1}$ from $\zeta^j$ by refining by ${\tt ref}\%$.
    \EndIf
    \If {${\tt stolb} < \eta_{S_1,j+1}< {\tt stola}$ and $counter > threshold$}
    \State Create $\zeta^{j+1}$ from $\zeta^j$ by refining by ${\tt ref}\%$ and coarsening by ${\tt coar}\%$.
    \EndIf
     \If {$\eta_{S_1,j+1} < {\tt stolb}$ and $counter > threshold$}
    \State Create $\zeta^{j+1}$ from $\zeta^j$ by coarsening by ${\tt coar}\%$.
    \EndIf
    \State Calculate $u_h^{j+1}$ from $u_h^j$.
    \State $j\leftarrow j+1$
    \State $time \leftarrow time + \tau_{j+1}$
      \If {$counter < threshold$}
    \State counter = 0
    \EndIf
    \EndWhile
  \end{algorithmic}
  \caption{Space-time adaptivity}
\end{algorithm}
\end{section}

Both mesh refinement and coarsening are driven by the term $\eta_{S_1,j+1}$. The size of the elemental contributions to $\eta_{S_1,j+1}$ determines the percentages of elements to be refined and/or coarsened depending on the two spatial tolerances {\tt stola} and {\tt stolb}. 
The nature of the time estimator, $\eta_T$, makes it difficult to use as a temporal refinement indicator so, to this end we define 
$\hat{\eta}_{T,j+1}$ given by
\begin{equation}
\hat{\eta}_{T,j+1}^2 = \frac{1}{4} \tau_{j+1} \eta_{T_1,j+1}^2  + \min\{\alpha_T,T\}\int_{t^j}^{t^{j+1}} \! \eta^2_{T_2,j+1} \, dt.
\end{equation}
The sum of all terms $\hat{\eta}_{T,j+1}$ bound $\eta_T$ and hence can be used to drive temporal refinement subject to a time tolerance {\tt ttol} on each time interval.

\begin{remark}
The term $ \eta_{S_2,j+1}$ is smaller than $\eta_{S_1,j+1}$ both from a theoretical and a practical standpoint. It is, therefore, not used to drive the spatial mesh modification.
\end{remark}

A difference of Algorithm 1 to (standard) adaptivity algorithms for conforming finite elements for parabolic problems is that here a limit is given on how much the mesh can change. This is introduced in order to to ensure termination of the algorithm.
The input value $m$ is a bound on how many times the mesh can change per unit time. The value of $m$ can be very large in practice.

\begin{section}{Numerical experiments}\label{numerics_sec}

We shall investigate numerically the presented a posteriori bounds and the performance of the adaptive algorithm through an implementation based on the {\tt deal.II} finite element library (\cite{BHK07}). All the numerical experiments have been performed using the  high performance computing facility ALICE at the University of Leicester. 

We denote by $\lambda_k$ the total number of degrees of freedom on the union mesh $\zeta^k \cup \zeta^{k+1}$. Hence, the weighted total number of degrees of freedom of the problem is given by
\begin{equation}
Total\mbox{ }DoFs := \sum_{j=0}^{n-1} \tau_{j+1}\lambda_j.
\end{equation}
In all examples presented below, unless otherwise stated, $\gamma=10$,  ${\tt stolb} = {\tt stola}/5$, ${\tt ref}\%=6.25\%$ and $\zeta^0$ is an $8\times8$ uniform quadrilateral mesh. We further take ${\tt coar} \% = 10 \%$ for Example 1, Example 2 and Example 4 while we set ${\tt coar} \% = 30 \%$ for Example 3.
\begin{subsection}{Example 1}

\noindent Let $\Omega = (0,1)^2$, ${\bf a}=(1,1)^T$, $b=0$, $u_0=0$, $T=10$ and select the function $f$ so that the exact solution to problem \eqref{model_weak} is given by
\begin{equation}
u(x,y,t)=(1-e^{-t})\bigg(\frac{e^{(x-1)/\varepsilon}-1}{e^{-1/\varepsilon}-1}+x-1 \bigg)\bigg(\frac{e^{(y-1)/\varepsilon}-1}{e^{-1/\varepsilon}-1}+y-1 \bigg).
\end{equation}
The solution exhibits boundary layers at the outflow boundary of the domain of width $\mathcal{O}(\varepsilon)$ as well as a temporal boundary layer.

We begin by assessing the decay of the estimators on uniform spatial and temporal meshes. To this end, the number of time-steps is doubled while the number of spatial degrees of freedom is quadrupled by respective uniform refinements. The results for the spatial polynomial degree $p=1$ are given in Tables \ref{ex1_opt1} and \ref{ex1_opt2}. As expected, the error halves (Err. Ratio) for each space-time uniform refinement and the same is observed for the a posteriori estimator (Est. Ratio), implying optimality of the a posteriori error estimator for this example. This is achieved immediately for $\varepsilon = 1$ and upon sufficient resolution of the boundary layer for $\varepsilon = 10^{-2}$. The a posteriori estimator appears to be of optimal rate under uniform refinement in all the numerical examples presented in this work; these results are omitted for brevity.

\begin{table}
\begin{center}
\begin{tabular}{||c|r|c|c|c|c||}
\hline
Time steps & Total DoFs & Est. & Est.Ratio & Error & Err.Ratio \\
\hline
10&	  160&  	3.83e-1	& $\ $ &	6.45e-2	 & $\ $ \\
20&	640&	 1.94e-1  &  0.507	& 3.07e-2 &	0.475 \\
40&	2560	&    9.79e-2	& 0.502&	 1.47e-2 &	0.479\\
80&	10240&	4.90e-2&	0.501&	 7.17e-3 &	0.487\\
160&	40960&	2.45e-2&	0.500&	3.53e-3&	0.492\\
320&	163840	& 1.22e-2&	0.500&	1.75e-3&	0.496\\
640&	655360	& 6.14e-3&	0.500&	8.72e-4&	0.497\\
1280	& 2621440&	3.07e-3&	0.500&	4.35e-4&	0.498\\
\hline
\end{tabular}
\caption{Example 1. Decay of the a posteriori error estimator for $p=1$, $\varepsilon=1$.} 
\label{ex1_opt1}
\end{center}
\end{table}

\begin{table}
\begin{center}
\begin{tabular}{||c|r|c|c|c|c||}
\hline
Time steps & Total DoFs & Est. & Est.Ratio & Error & Err.Ratio \\
\hline
10&	  160&  	1.66e+1	& $\ $ &	8.56e-1	 & $\ $ \\
20&	640&	     1.29e+1  &  0.773	& 1.06e+0 &	1.240 \\
40&	2560	&       1.08e+1	& 0.838&	 1.18e+0 &	1.114\\
80&	10240&	   8.68e+0&	0.802&	 1.27e+0 &	1.078\\
160&	40960&	5.91e+0&	0.680&	1.14e+0&	    0.899\\
320&	163840	&  3.47e+0&	0.587&	 7.54e-1&	   0.657\\
640&	655360	&  1.88e+0&	0.542&	 4.06e-1&	   0.538\\
1280	& 2621440&	 0.98e+0&	0.520&	2.01e-1&	0.496\\
\hline
\end{tabular}
\caption{Example 1. Decay of the a posteriori error estimator for $p=1$, $\varepsilon=10^{-2}$.} 
\label{ex1_opt2}
\end{center}
\end{table}
 
We now assess the performance of the space-time adaptive algorithm presented in Section~\ref{adaptive_sec}; we do this for $p>1$ as the case $p=1$ needs special care and hence will be discussed separately.

We begin by fixing a temporal tolerance that produces enough time steps so that the temporal contribution to the error is very small in comparison to the spatial contribution. The spatial tolerance is then gradually reduced to compare the rates of convergence of the spatial estimator with that of the actual total error. Results for $\epsilon=1$ and $\epsilon=10^{-2}$ and for polynomial degrees $p=2$ and $p=3$ are shown in Figure~\ref{Example1_errors_eff}; here as in all the following convergence plots an unmarked line indicates the optimal order of convergence.
Optimal rates of convergence are observed with respect to the total degrees of freedom for both the estimator and the error. Further, as  shown in the bottom plots of Figure~\ref{Example1_errors_eff}, the effectivity indices 
are bounded and remain between 6 and 11 for the two values of $\epsilon$; these are directly comparable to those observed in \cite{SZ09} for the stationary problem. 
\begin{figure}
\centering

\includegraphics[scale=.5]{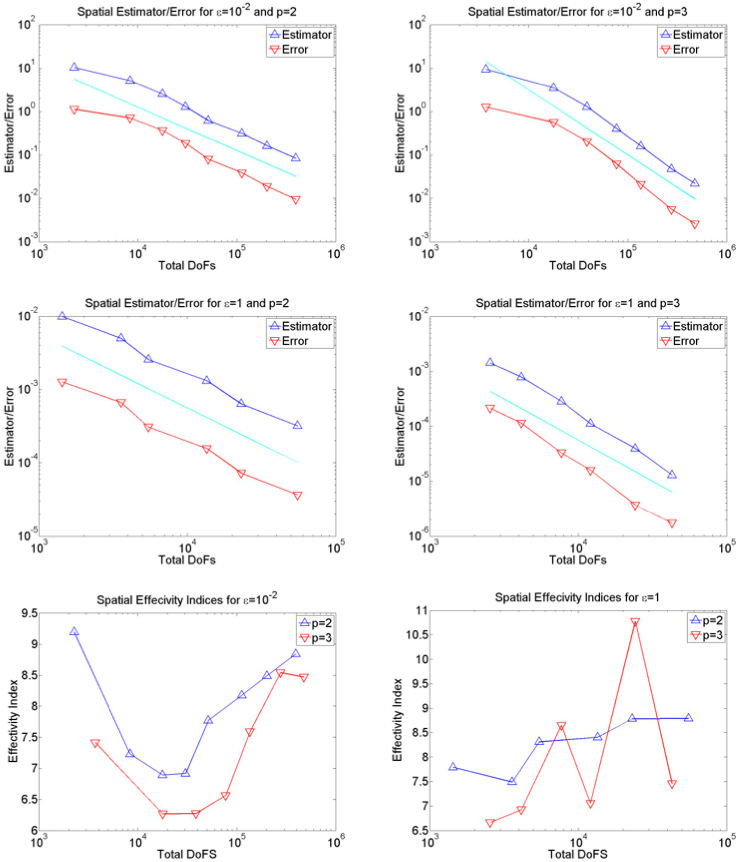}

\caption{Example 1: Spatial rates of convergence and effectivity indices for $\epsilon=1,\epsilon=10^{-2}$ and $p=2,3$ }
\label{Example1_errors_eff}
\end{figure}
\begin{figure}
\centering

\includegraphics[scale=.35]{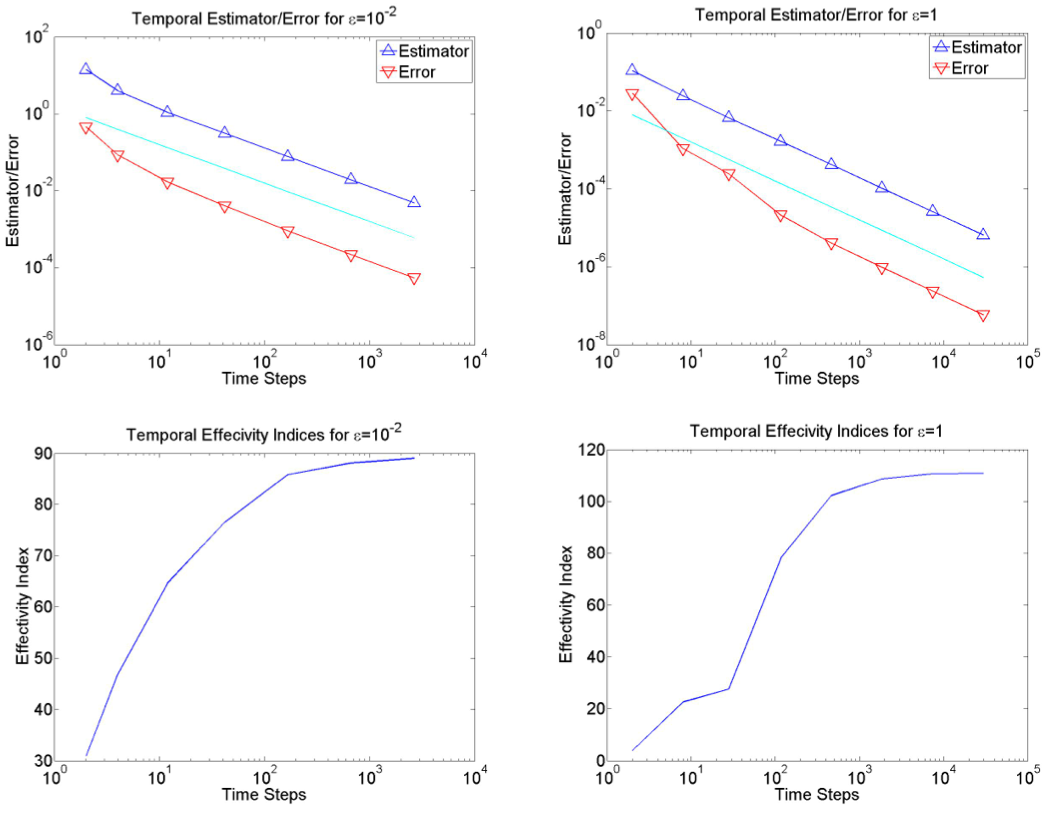}

\caption{Example 1: Temporal rates of convergence and effectivity indices for $\epsilon=1,10^{-2}$.}
\label{Example1_temp_errors_eff}
\end{figure}
\begin{figure}
\centering

\includegraphics[scale=.35]{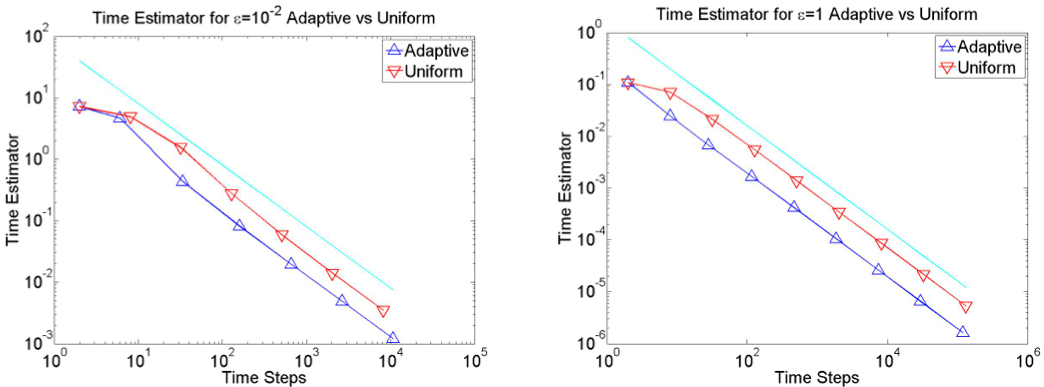}

\caption{Example 1:  Adaptive vs uniform comparison for $\epsilon=1,10^{-2}$.}
\label{Example1_temp_errors_eff_compare}
\end{figure}

In order to study the temporal rates of convergence of the adaptive algorithm any boundary layers must be fully resolved so that the spatial error is dominated by the temporal one. To this end, we set $\gamma=100$ and use $p=10$ on a sufficiently fine fixed spatial mesh. We monitor the algorithm's behaviour upon gradually reducing the temporal tolerance. 
The results are given in Figure \ref{Example1_temp_errors_eff}.
Optimal order is observed with respect to both the time estimator and the error. The effectivity indices appear to remain bounded and converge to a smaller value for $\epsilon=10^{-2}$ than for $\epsilon=1$ which is evidence of robustness with respect to the P\'eclet number.

The presence of a temporal boundary layer in the solution motivates a comparison between adaptive and uniform time-stepping. We do this by using $p=3$ and a sufficiently resolved spatial mesh and by reducing the temporal tolerance in an adaptive fashion and uniformly, respectively. 
The results are given in the plots in Figure \ref{Example1_temp_errors_eff_compare}: the savings achieved by using adaptive time-stepping are evident.

Let us now discuss the behaviour of the space-time a posteriori estimator in the case $p=1$. Recall that the estimator has already been proven to be of optimal rate in both space and time under uniform refinement, {\em cf.} Tables~\ref{ex1_opt1} and~\ref{ex1_opt2}.
In the adaptive setting, however, where the mesh can change between time steps, we observe that the rate of convergence of  the time estimator for $p=1$ may be slower than the expected optimal rate. This is indeed the case for this test problem, as shown in the right plot of Figure~\ref{ex1_p1} by comparison with the results obtained with $p=2$ and $p=3$.
On the other hand, the spatial estimator retains optimal rate, as shown in the left plot of Figure~\ref{ex1_p1}.
\begin{figure}
\centering

\includegraphics[scale=.35]{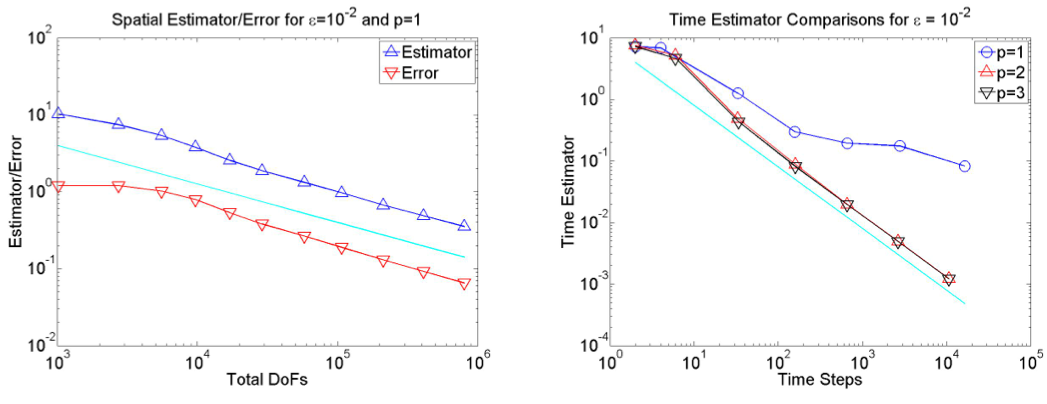}

\caption{Example 1: Spatial and temporal rates of convergence  for $\epsilon=10^{-2}$ and $p=1$.}
\label{ex1_p1}
\end{figure}

Notice that the method itself could be suboptimal under mesh change for $p=1$; this is something we have not been able to test and indeed the nature of this degeneracy  remains unclear to the authors. It is also unclear if this is, or not, a special feature of the backward-Euler-dG method considered in this work. In the literature, numerical experiments including \emph{rates} of convergence of space-time adaptive algorithms for parabolic problems are very rare, especially for non-conforming methods. Therefore, it is not possible at this point to assess the generality of this phenomenon. We mention in passing the recent work \cite{BKM12} where degeneracy due to mesh modification (including mesh refinement) has been observed for Crank-Nicolson FEM discretisations. 
As shown above, the degeneracy of the temporal rate of convergence when spatial mesh-change is present disappears for $p>1$. Hence, for the remainder of this work, all the numerical experiments with space-time adaptivity are performed with $p>1$.

\end{subsection}

\begin{subsection}{Example 2}

We set $\Omega=(-1,1)^2$, ${\bf a}=(1,1)^T$, $b=1$, $f=sin(5t)xy$, $u_0=0$ and $T = 2\pi$. The solution exhibits layers of width $\mathcal{O}(\varepsilon)$ in the proximity of the outflow boundary and is oscillatory in time. The sharpness of the boundary layers depend on time, thus making this a good test of the ability of the algorithm to add and remove degrees of freedom.
\begin{figure}
\centering

\includegraphics[scale=.35]{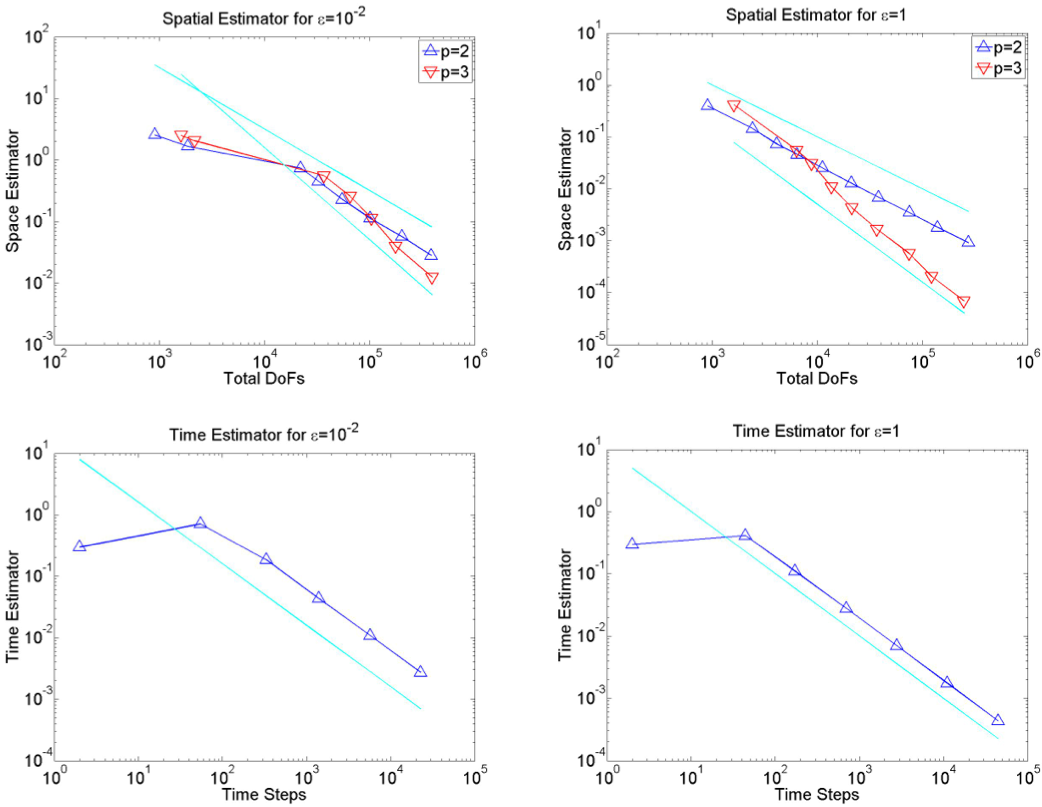}

\caption{Example 2: Spatial and temporal rates of convergence for $\varepsilon=1,10^{-2}$.}
\label{Example2_rates}
\end{figure}

As in Example 1, we begin by fixing a temporal tolerance while decreasing the spatial tolerance to observe the rates of convergence for the space estimator. We then set $p=3$ with a spatial tolerance small enough to resolve any boundary layers, while reducing the temporal tolerance to observe the rates of the time estimator. The results are displayed in Figure \ref{Example2_rates}. Optimal rates of convergence are observed for both the space and the time estimators.

To assess the mesh change driven by the adaptive algorithm we also plot the individual degrees of freedom on each mesh against time for a given spatial tolerance and temporal tolerance. The results are given in Figure \ref{ex2_dof_time}. We observe that the adaptive algorithm is adding and removing degrees of freedom at a rate that is in accordance with the oscillating nature of the solution driven by the sinusoidal forcing function $f$. 

\begin{figure}
\centering

\includegraphics[scale=.35]{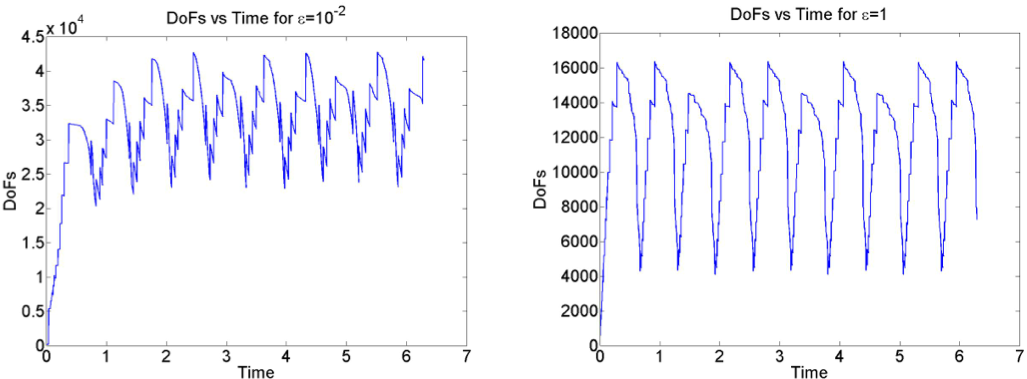}

\caption{Example 2: Degrees of freedom on each time step plotted against time for $\varepsilon=1,10^{-2}$}
\label{ex2_dof_time}
\end{figure}
\end{subsection}

\begin{subsection}{Example 3}
 Let $\Omega=(-1,1)^2$,  $T = 100$, ${\bf a}=(y,-x)^T$, $b=0$, $f=0$, and $u_0=e^{-64(x-0.5)^2}e^{-64(y-0.5)^2}$.
  The PDE convects the initial two dimensional Gaussian profile along the circular wind while diffusing it at a rate depending upon $\varepsilon$.
 
\begin{figure}
\centering

\includegraphics[scale=.35]{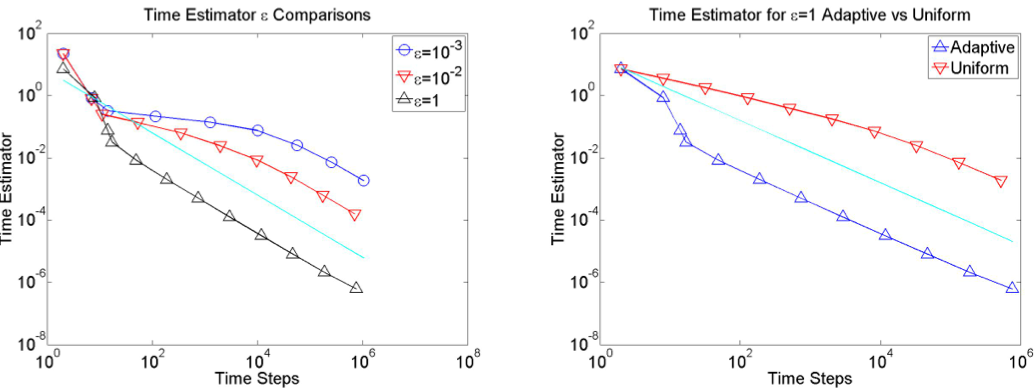}

\caption{Example 3: Convergence of the time estimator for $\varepsilon=1,10^{-2},10^{-3}$ and comparison of adaptive vs. uniform time steps.}
\label{Example3_time}
\end{figure}

Setting $p=3$ and fixing a spatial tolerance and an initial tolerance, we reduce the temporal tolerance to observe the temporal rates of the problem. Additionally, for $\epsilon=1$, we compare adaptive time-stepping and uniform time-stepping. The results are given in Figure \ref{Example3_time}. Some meshes at various time steps  produced by the algorithm for $\epsilon=10^{-3}$ are shown in Figure \ref{Example3_meshes}.
\begin{figure}[!]
\centering

\includegraphics[scale=.35]{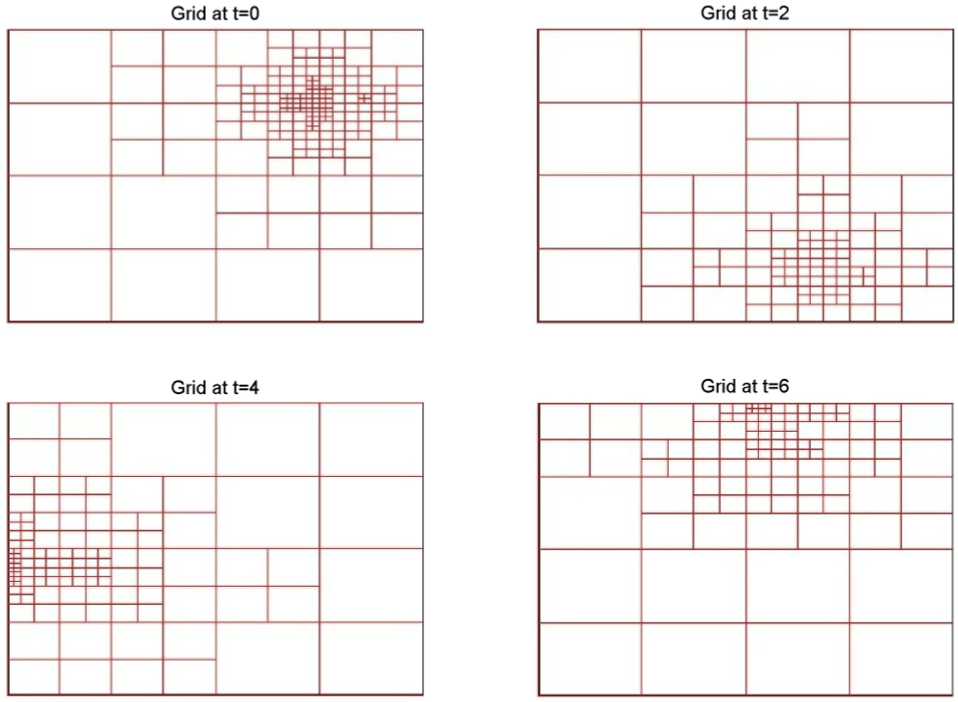}

\caption{Example 3: grid snapshots.}
\label{Example3_meshes}
\end{figure}

The smaller the value of $\epsilon$, the longer it takes the time estimator to reach optimal order which is to be expected theoretically. Furthermore, large savings are observed using adaptive time-stepping over uniform time-stepping when the diffusion is large; time steps would be otherwise wasted when the solution is effectively zero.
\end{subsection}
\begin{subsection}{Example 4}
Let $\Omega=(0,1)^2$, ${\bf a}=(sin(t),cos(t))^T$, $b=0$, $f=1$, $ u_0=0$ and $T = 2\pi$. The nature of the solution is rather uniform in time but has a boundary layer of width $\mathcal{O}(\varepsilon)$ whose location does depend on time. Therefore, this example is well suited to test the ability of the algorithm to adapt the grid to the feature of the solution as time evolves.

 As in previous examples, we fix a temporal tolerance and then reduce the spatial tolerance to observe the rates of the space estimator. Again, we also fix $p=3$ and a spatial tolerance small enough to ensure that all boundary layers are sufficiently resolved and then observe the rates of the time estimator. These results are given in Figure \ref{ex4_rates}. Finally, grids at various times for $\epsilon=10^{-2}$ and $p=2$ are given in Figure \ref{ex4_grids}.
Optimal spatial and temporal rates of convergence are observed. 
The grids produced for $\epsilon=10^{-2}$ clearly show that the adaptive algorithm is picking up the boundary layers as they move around the domain and that unneeded degrees of freedom are not retained.
\begin{figure}
\centering

\includegraphics[scale=.35]{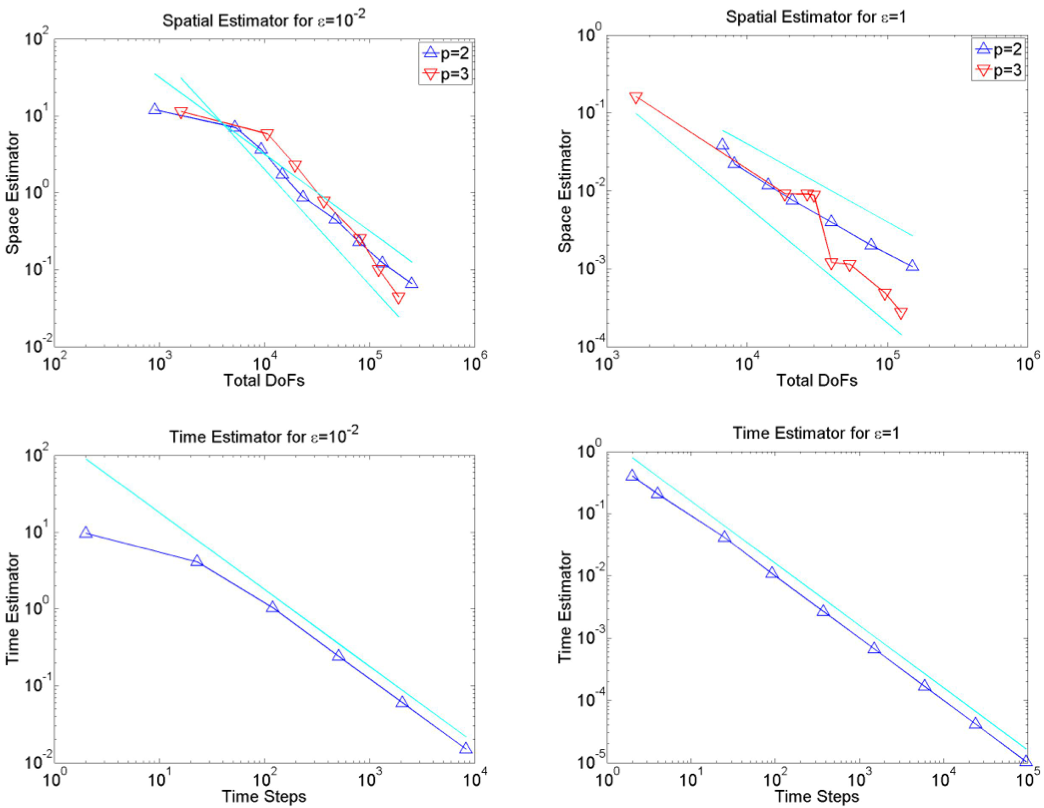}

\caption{Example 4: Spatial and temporal rates of convergence for $\varepsilon=1,10^{-2}$.}
\label{ex4_rates}\vspace{-.5 cm}

\end{figure}
\begin{figure}
\centering

\includegraphics[scale=.35]{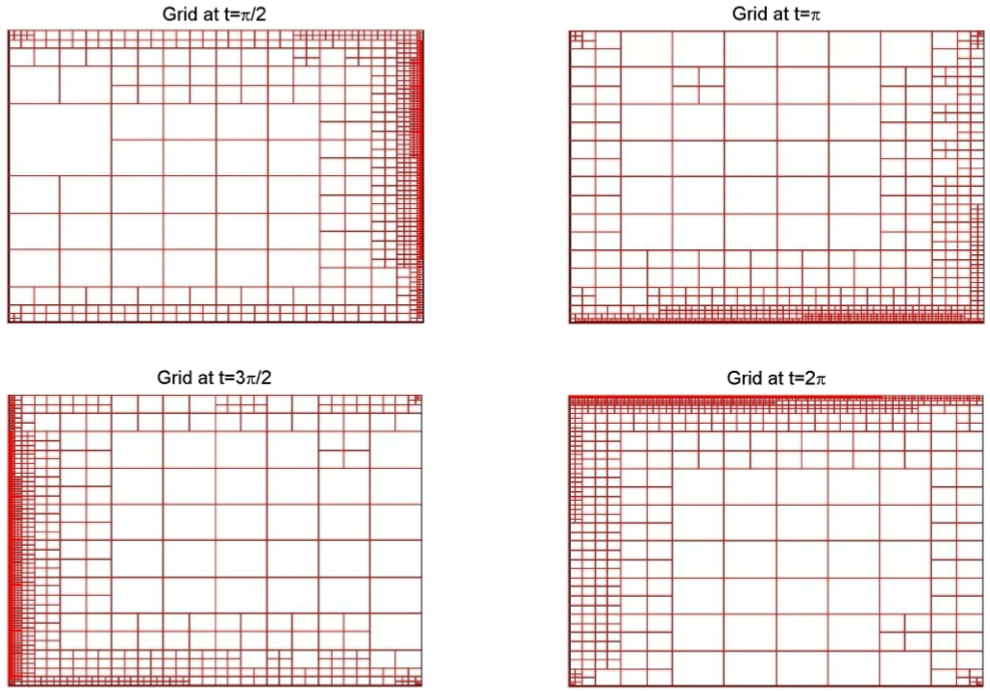}

\caption{Example 4: grid snapshots.}
\label{ex4_grids}\vspace{-.2 cm}
\end{figure}

\end{subsection}

\end{section}

\begin{section}{Conclusions}\label{conclusions_sec}

An a posteriori error estimator for the discontinuous Galerkin spatial discretisation of a non-stationary linear convection-diffusion equation is presented. The derivation of the estimator is based on reconstruction techniques to make use of robust a posteriori estimators for elliptic problems already developed in the literature. Our numerical examples clearly indicate that the error estimator is robust and the respective space-time adaptive algorithm works well for the studied problems. The spatial effecitivity indices are in an identical range to those observed in (\cite{SZ09}) and the temporal effectivity indices are approximately of the same order as those currently seen in the literature for similar problems (\cite{GLV11}) and appear independent of $\epsilon$. Further study to understand the temporal rate of convergence for $p=1$ under mesh-modification as well as extensions to high order time-stepping methods will be considered in the future. 
\end{section}

\bibliographystyle{plain}
\bibliography{bibliography}

\begin{thebibliography}{10}

\bibitem{AO00}
Mark Ainsworth and J.~Tinsley Oden.
\newblock {\em A posteriori error estimation in finite element analysis}.
\newblock Pure and Applied Mathematics (New York). Wiley-Interscience [John
  Wiley \& Sons], New York, 2000.

\bibitem{ABR05}
Rodolfo Araya, Edwin Behrens, and Rodolfo Rodr{\'{\i}}guez.
\newblock An adaptive stabilized finite element scheme for the
  advection-reaction-diffusion equation.
\newblock {\em Appl. Numer. Math.}, 54(3-4):491--503, 2005.

\bibitem{APS05}
Rodolfo Araya, Abner~H. Poza, and Ernst~P. Stephan.
\newblock A hierarchical a posteriori error estimate for an
  advection-diffusion-reaction problem.
\newblock {\em Math. Models Methods Appl. Sci.}, 15(7):1119--1139, 2005.

\bibitem{BHK07}
W.~Bangerth, R.~Hartmann, and G.~Kanschat.
\newblock deal.{II}---a general-purpose object-oriented finite element library.
\newblock {\em ACM Trans. Math. Software}, 33(4):Art. 24, 27, 2007.

\bibitem{BKM12}
Eberhard; B\"ansch, Foteini; Karakatsani, and Charalambos Makridakis.
\newblock A posteriori error control for fully discrete {C}rank-{N}icolson
  schemes.
\newblock {\em SIAM J. Numer. Anal.}, page to appear, 2012.

\bibitem{BHL03}
Roland Becker, Peter Hansbo, and Mats~G. Larson.
\newblock Energy norm a posteriori error estimation for discontinuous
  {G}alerkin methods.
\newblock {\em Comput. Methods Appl. Mech. Engrg.}, 192(5-6):723--733, 2003.

\bibitem{BC04}
Stefano Berrone and Claudio Canuto.
\newblock Multilevel a posteriori error analysis for
  reaction-convection-diffusion problems.
\newblock {\em Appl. Numer. Math.}, 50(3-4):371--394, 2004.

\bibitem{CF04}
Zhiming Chen and Jia Feng.
\newblock An adaptive finite element algorithm with reliable and efficient
  error control for linear parabolic problems.
\newblock {\em Math. Comp.}, 73(247):1167--1193 (electronic), 2004.

\bibitem{D82}
Todd Dupont.
\newblock Mesh modification for evolution equations.
\newblock {\em Math. Comp.}, 39(159):85--107, 1982.

\bibitem{EP05}
Alexandre Ern and Jennifer Proft.
\newblock A posteriori discontinuous {G}alerkin error estimates for transient
  convection-diffusion equations.
\newblock {\em Appl. Math. Lett.}, 18(7):833--841, 2005.

\bibitem{EV10}
Alexandre Ern and Martin Vohral{\'{\i}}k.
\newblock A posteriori error estimation based on potential and flux
  reconstruction for the heat equation.
\newblock {\em SIAM J. Numer. Anal.}, 48(1):198--223, 2010.

\bibitem{GHH07}
Emmanuil~H. Georgoulis, Edward Hall, and Paul Houston.
\newblock Discontinuous {G}alerkin methods for advection-diffusion-reaction
  problems on anisotropically refined meshes.
\newblock {\em SIAM J. Sci. Comput.}, 30(1):246--271, 2008.

\bibitem{GLV11}
Emmanuil~H. Georgoulis, Omar Lakkis, and Juha~M. Virtanen.
\newblock A posteriori error control for discontinuous {G}alerkin methods for
  parabolic problems.
\newblock {\em SIAM J. Numer. Anal.}, 49(2):427--458, 2011.

\bibitem{HSW07}
Paul Houston, Dominik Sch{\"o}tzau, and Thomas~P. Wihler.
\newblock Energy norm a posteriori error estimation of {$hp$}-adaptive
  discontinuous {G}alerkin methods for elliptic problems.
\newblock {\em Math. Models Methods Appl. Sci.}, 17(1):33--62, 2007.

\bibitem{HS01}
Paul Houston and Endre S{\"u}li.
\newblock Adaptive {L}agrange-{G}alerkin methods for unsteady
  convection-diffusion problems.
\newblock {\em Math. Comp.}, 70(233):77--106, 2001.

\bibitem{KP03}
Ohannes~A. Karakashian and Frederic Pascal.
\newblock A posteriori error estimates for a discontinuous {G}alerkin
  approximation of second-order elliptic problems.
\newblock {\em SIAM J. Numer. Anal.}, 41(6):2374--2399 (electronic), 2003.

\bibitem{K03}
Gerd Kunert.
\newblock A posteriori error estimation for convection dominated problems on
  anisotropic meshes.
\newblock {\em Math. Methods Appl. Sci.}, 26(7):589--617, 2003.

\bibitem{LPP09}
Alexei Lozinski, Marco Picasso, and Virabouth Prachittham.
\newblock An anisotropic error estimator for the {C}rank-{N}icolson method:
  application to a parabolic problem.
\newblock {\em SIAM J. Sci. Comput.}, 31(4):2757--2783, 2009.

\bibitem{MN03}
Charalambos Makridakis and Ricardo~H. Nochetto.
\newblock Elliptic reconstruction and a posteriori error estimates for
  parabolic problems.
\newblock {\em SIAM J. Numer. Anal.}, 41(4):1585--1594, 2003.

\bibitem{RST08}
Hans-G{\"o}rg Roos, Martin Stynes, and Lutz Tobiska.
\newblock {\em Robust numerical methods for singularly perturbed differential
  equations}, volume~24 of {\em Springer Series in Computational Mathematics}.
\newblock Springer-Verlag, Berlin, second edition, 2008.
\newblock Convection-diffusion-reaction and flow problems.

\bibitem{S08}
Giancarlo Sangalli.
\newblock Robust a-posteriori estimator for advection-diffusion-reaction
  problems.
\newblock {\em Math. Comp.}, 77(261):41--70 (electronic), 2008.

\bibitem{SS05}
Alfred Schmidt and Kunibert~G. Siebert.
\newblock {\em Design of adaptive finite element software}, volume~42 of {\em
  Lecture Notes in Computational Science and Engineering}.
\newblock Springer-Verlag, Berlin, 2005.
\newblock The finite element toolbox ALBERTA, With 1 CD-ROM (Unix/Linux).

\bibitem{SZ09}
Dominik Sch{\"o}tzau and Liang Zhu.
\newblock A robust a-posteriori error estimator for discontinuous {G}alerkin
  methods for convection-diffusion equations.
\newblock {\em Appl. Numer. Math.}, 59(9):2236--2255, 2009.

\bibitem{SZ11}
Dominik Sch{\"o}tzau and Liang Zhu.
\newblock A robust {\it a posteriori} error estimate for {$hp$}-adaptive {DG}
  methods for convection-diffusion equations.
\newblock {\em IMA J. Numer. Anal.}, 31(3):971--1005, 2011.

\bibitem{S12}
Ivana {\v S}ebestov\'a.
\newblock Two-sided a posteriori error estimates for the {DGM}s for the heat
  equation.
\newblock In {\em In A. Cangiani, R.L. Davidchack, E.H. Georgoulis, A. Gorban,
  J. Levesley, M. Tretyakov (eds.), Numerical Mathematics and Advanced
  Applications, Proceedings of the ENUMATH 2011 Conference, Leicester.}
  Springer, 2012.

\bibitem{SW06}
Shuyu Sun and Mary~F. Wheeler.
\newblock A posteriori error estimation and dynamic adaptivity for symmetric
  discontinuous {G}alerkin approximations of reactive transport problems.
\newblock {\em Comput. Methods Appl. Mech. Engrg.}, 195(7-8):632--652, 2006.

\bibitem{V98}
R.~Verf{\"u}rth.
\newblock A posteriori error estimators for convection-diffusion equations.
\newblock {\em Numer. Math.}, 80(4):641--663, 1998.

\bibitem{V05}
R.~Verf{\"u}rth.
\newblock Robust a posteriori error estimates for nonstationary
  convection-diffusion equations.
\newblock {\em SIAM J. Numer. Anal.}, 43(4):1783--1802 (electronic), 2005.

\bibitem{V05II}
R.~Verf{\"u}rth.
\newblock Robust a posteriori error estimates for stationary
  convection-diffusion equations.
\newblock {\em SIAM J. Numer. Anal.}, 43(4):1766--1782 (electronic), 2005.

\bibitem{V96}
R{\"u}diger Verf{\"u}rth.
\newblock {\em A Review of {\em A Posteriori} Error Estimation and Adaptive
  Mesh-Refinement Techniques}.
\newblock Wiley-Teubner, Chichester-Stuttgart, 1996.

\end{thebibliography}
\end{document}